\documentclass[12pt]{article}
\usepackage[english]{babel}
\usepackage{amsmath,amssymb,amsthm,xspace}

\RequirePackage{ifpdf}
\ifpdf
   \usepackage[pdftex]{hyperref}
\else
   \usepackage[hypertex]{hyperref}
\fi

\title{Sharply $3$-transitive groups}
\author{Katrin Tent\footnote{Partially supported by SFB 878}}
\date{\today}

\newtheorem{theorem}{Theorem}[section]

\newtheorem{lemma}[theorem]{Lemma}

\newtheorem{corollary}[theorem]{Corollary}
\newtheorem{definition}[theorem]{Definition}
\newtheorem{remark}[theorem]{Remark}
\newtheorem{normalform}[theorem]{Normal forms}

\newcommand{\nc}{\newcommand}
\nc{\inv}{^{-1}}
\nc{\Q}{\mathbb{Q}}
\nc{\R}{\mathbb{R}}
\nc{\F}{\mathbb{F}}
\nc{\C}{\mathcal{C}}
\nc{\G}{\mathcal{G}}
\nc{\M}{\mathcal{M}}
\nc{\U}{\mathbb{U}}
\nc{\Frl}{Fra\"iss\'e limit\xspace}
\nc{\Frls}{Fra\"iss\'e limits\xspace}
\nc{\fg}{finitely generated\xspace}

\renewcommand{\phi}{\varphi}

\newcommand{\Ind}{
 \setbox0=\hbox{$x$}\kern\wd0\hbox to 0pt{\hss$
 \mid$\hss}\lower.9\ht0\hbox to 0pt{\hss$\smile$\hss}\kern\wd0
}

\begin{document}
\maketitle
\begin{abstract}We construct the first sharply $3$-transitive groups not arising from a near field, i.e. point stabilizers have no
nontrivial abelian normal subgroup. 
\end{abstract}

\section{Introduction}

The finite sharply $2$- and $3$-transitive groups were classified by Zassenhaus in \cite{Z1} and \cite{Z2} 
in the 1930's and were shown to arise
from so-called near-fields. They essentially look like the groups
of affine linear transformations $x\mapsto ax+b$ or
Moebius transformations $x\mapsto \frac{ax+b}{cx+d}$, respectively.

It remained an open problem whether the same
holds for infinite sharply $2$- and $3$-transitive 
groups and much literature
on this topic is available, see \cite{RST} for background and more recent references. 
In \cite{RST} the first construction of sharply $2$-transitive groups without any nontrivial abelian normal subgroup is given. 

However the question remained open whether the groups constructed
there can be extended to groups acting sharply $3$-transitively. We here
use a different approach directly constructing sharply $3$-transitive groups
using partial group actions. These are the first known examples of
sharply $3$-transitive groups whose point stabilizers have no non-trivial
abelian normal subgroup and thus do not arise from near-fields. 

By results of Tits \cite{Tits} and Hall \cite{hall} there are no infinite sharply $k$-transitive
groups for $k\geq 4$, see e.g. \cite{DM}.

\section{The main theorem}

For brevity we call an element of order $3$ a $3$-cycle and we
say that a group action is $3$-sharp if all $3$-point stabilizers are trivial.

\begin{theorem}\label{t:main}
Let $G_0$ be a group in which all $3$-cycles and all involutions, respectively, are conjugate and such that there exists a $3$-cycle $a$ and an involution $t$
in $G_0$ with $\langle a,t\rangle\cong S_3$. 
Assume that $G_0$ acts on a set $X_0$ in such a way that
  \begin{enumerate}
  \item the action is $3$-sharp;
  \item the involution $t\in G_0$ fixes a unique point $x_0\in X_0$;
  \item the $3$-cycle $a\in G_0$ is fixed point free;
  \item we have  $(x_0,x_0a,x_0a^2)=(x_0,x_0a,x_0at)$;
  \item if $B=(x,y,z)$ is a triple in $X_0$ such that the setwise
  stabilizer of $B$ in $G_0$ is isomorphic to $S_3$, 
  then there is some $g\in G_0$
  with $Ag=B$ where $A=(x_0,x_0a,x_0a^2)=(x_0,x_0a,x_0at)$.
  \end{enumerate}
  Then we can extend $G_0$ to a sharply $3$-transitive action of 
  \[G=\bigl(   (\langle a\rangle\times F(U))\ast  _{\langle a\rangle} G_0\ast_{\langle t\rangle}\left((\langle t\rangle\times F(S))\right)\bigr)\ast F(R)\]
on a suitable set $Y\supset X_0$, where $F(R), F(S), F(U)$ are free groups on disjoint sets $R, S,U$ with $|R|,|S|, |U|= \max\{|G_0|,\aleph_0\}$.
\end{theorem}

\begin{remark}
Note that $S_3$ in its natural action on three elements satisfies
the assumptions of Theorem~\ref{t:main} 
(as does $PGL(2,2)\cong S_3$ acting as a subgroup of
$PGL(2,2^3)$ on the projective line over $\F_{2^3}$.)
\end{remark}

Therefore we have

\begin{corollary}
There exist groups $G$ acting sharply $3$-transitively on some set $X$
such that for $x\in X$ the point stabilizer $G_x$ of $x$ has no nontrivial abelian normal subgroup.
\end{corollary}

\begin{proof}
Applying Theorem~\ref{t:main} to $G_0=S_3=\langle a\rangle\rtimes \langle t\rangle$
we see that no two distinct involutions in $G$ commute.
The point stabilizer $G_x$ acts sharply $2$-transitively on $X\setminus\{x\}$. Since the involutions in $G_0$ and hence (by conjugacy) also in $G$ have unique fixed points,
the involutions of $G_x$ in their action on $X\setminus\{x\}$ are fixed point free, so the sharply $2$-transitive group $G_x$ is said to have characteristic $2$.  By \cite{Neumann} (see also e.g. \cite{BN}, 11.46) a nontrivial  abelian normal subgroup of $G_x$
would have to consist of elements of order $2$, each being the product of two involutions in $G_x$. From the construction of $G$, it can be seen that this is impossible in $G$.
\end{proof}
 
\section{The construction}

In this section we prove Theorem~\ref{t:main}, i.e. we construct
$G$ and its action on a set $X$ from $G_0$ and $X_0$. We continue to use
the notation introduced in the previous section.

For the construction we use partial group actions as in \cite{TZ}, considered also by Rips and Segev. The
construction proceeds by extending inductively both the group $G_0$ and
the underlying set $X_0$ in such a way that the assumptions of  Theorem~\ref{t:main} are preserved. This is done in two separate kinds of extensions: on the one hand we extend the group $G_0$ in order to make the action a bit more $3$-transitive. On the other hand we extend the underlying set $X_0$ in order to let the extended group act, forcing us in turn to extend the
group again etc. In the limit, the group $G$ will be sharply $3$-transitive on a set $X$ containing $X_0$.

\begin{definition}\label{d:partial}
  A partial action of $G$ on a set $X$ containing $X_0$  consists of an action of $G_0$
  on $X$ and partial actions of the generators in $S\cup R\cup U$
  such that
  \begin{enumerate}\item for $s\in S$ we have $x_0s=x_0$ and  if
    $xs$ is defined for $x\in X\setminus\{x_0\}$, then so is $(xt)s$ and we have
    $(xt)s=(xs)t$.
    \item for $u\in U, x\in X$, if $xu$ is defined, then so are $(xa)u$ and
    $(xa^2)u$ and we have $(xa)u=(xu)a$  and $(xa^2)u=(xu)a^2$.
\end{enumerate}
\end{definition}

We now define normal forms suitable for our purpose:

\noindent
\begin{normalform}\label{d:normalform}{\rm Any element of $G$ can be written (not necessarily uniquely) 
as a reduced word in the generators
$G_0\setminus\{1\}\cup S\cup S^{-1}\cup R\cup R^{-1}\cup U\cup U^{-1}$  where we
say that a word is reduced if there are no subwords of the form
\begin{itemize}
\item $ff^{-1}$ for $f\in S\cup S^{-1}\cup R\cup R^{-1}\cup U\cup U^{-1}$,
\item $gh$ for $g,h\in G_0\setminus\{1\}$;
\item $s_1ts_2$ where $s_1,s_2\in S\cup S^{-1}$;
\item $ts_1\cdots s_ng$ (or its inverse) where $s_i\in S\cup S^{-1}, i=1,\ldots, n, g\in  G_0\setminus\{1\}$;
\item $u_1a^{\pm 1}u_2$ where $u_1,u_2\in U\cup U^{-1}$;
\item   $a^{\pm 1}u_1\cdots u_ng$ or its inverse where $u_i\in U\cup\/U^{-1}, i=1,\ldots, n, g\in  G_0\setminus\{1\}$.
\end{itemize}  
  The word $w$ is called \emph{cyclically reduced} if $w$ and every cyclic
  permutation of $w$ is reduced.}
\end{normalform}

 We say that for a word 
$w=s_1\cdots s_n$ in $G_0\setminus\{1\}\cup S\cup S^{-1}\cup R\cup R^{-1}\cup U\cup U^{-1}$, the element
$xw$ is \emph{defined} for $x\in X$ if for all initial segments of $w$ the action
on $x$ is defined, i.e. $xs_1,\ldots, (\ldots (xs_1)\ldots )s_i, i\leq n$, are defined and we write $xw=(\ldots (xs_1)\ldots )s_n$.

Notice that for elements from $G_0$ the action on $X$ is defined everywhere. Hence if $w, w'$ are reduced words with $w=w'$ in $G$ and $xw$ is defined, then
$xw'$ is defined as well and $xw=xw'$. Thus the expression $xg=y$ makes sense for $g\in G, x,y\in X$. 

If $G$ acts partially on $X$, then there is a canonical partial action on the set of triples
\[(X)^3 = \{(x,y,z)\in X^3\mid |\{x,y,z\}|=3\}.\]

\noindent
{\bf Terminology and notation}
For a triple $C=(x,y,z)$ we say that $g\in G$ \emph{shifts} $C$ if $Cg=(y,z,x)$ or
$Cg=(z,x,y)$ and we say that $h\in G$ \emph{flips} $C$ if
$Ch=(x,z,y)$ or a shift of this.
Note that if there are such elements $g,h\in G$ then the setwise
stabilizer of $C$ in $G$ is isomorphic to $S_3$.
We also say that an element $g\in G$ leaves a triple $C$ \emph{invariant}
if $Cg$ is defined and is equal to $C$ as a set. In this case we call $C$
a $g$-triple.  We call a triple \emph{free} if the only element
leaving it invariant is the identity.

We call two triples $C$ and $C'$ \emph{connected} if there is $w\in G$ such that $Cw$ is defined and equals $C'$. 
If $w\in G$ leaves a triple $C$ invariant, then writing $w=s_1\ldots s_n$
in reduced form, the triples $(C,Cs_1,C s_1s_2,\ldots, Cw)$ 
form a cycle which we call \emph{braided} if $Cw$ is equal to
$C$ as a set, but not necessarily as an ordered triple.

\begin{definition}\label{d:good}
 We call a partial action of $G$ on  $X$ \emph{good} (and say that
 $G$ \emph{acts well} on $X$) if 
\begin{enumerate}
\item the $3$-cycle $a$ acts without fixed point on $X$;
\item the involution $t$ has a unique fixed point, namely $x_0\in X$;

\hspace{-1.2cm} and for all triples
  $C\in (X)^3$ and $g,h\in G$ the following holds:
  \item $Cg=C$ implies $g=1$. 
  \item If $h$ flips $C$, then $h$ is a conjugate of $t$.
  \item If $g$ shifts $C$,  then $g$ is a conjugate of $a$.
  \item If $g$ shifts $C$ and $h$ flips $C$, then there is some
         $g'\in G$ such that $Cg'=A$ where  $A=(x_0,x_0a,x_0a^2)=(x_0,x_0a,x_0at)$.
  \end{enumerate}
\end{definition}

Note that the original action of $G_0$ on $X=X_0$ is good
and that $A$ is the unique triple of $X_0$
invariant under both $a$ and $t$. In order to make the action $3$-transitive
it suffices to connect $A$ to any other triple of the underlying set.
Notice that since $a$ does not fix any point, we have $(x,xa,xa^2)\in (X)^3$ for all $x\in X$.

In the remainder of the section we extend a good partial action in
two ways: by letting free generators take the (distinguished) triple $A$
to all $a$-, $t$- and free triples in order to eventually make the action $3$-transitive, and by extending the domain of the partial action
of a free generator in order to eventually make the action total.
We start with the last one:

\begin{lemma}[Extending the free generators]\label{l:ext free}
  Assume that $G$ acts well on $X$ and that for some $x\in X$ and $f\in S\cup S^{-1}\cup U\cup U^{-1}\cup R\cup R^{-1}$
  the expression $xf$ is not defined.  (Then  $xtf$ is not defined if $f\in S\cup S^{-1}$ 
  and  $xaf, xa^2f$ are not defined in case $f\in U\cup U^{-1}$.)
  
   Let $x'G_0=\{x'g_0\mid g_0\in
  G_0\}$ be a set of new elements on which  $G_0$ acts regularly  and extend the
  partial operation of $G$ to $X'=X\cup x'G_0$ by putting
  
  \begin{enumerate}
  \item  $xf=x'$ if $f\in R\cup R^{-1}$;
  \item  $xf=x'$ and $(xt)f=x't$ (and $x_0f=x_0$) 
  if $f\in S\cup S^{-1}$;
  \item  $xf=x'$, $(xa)f=x'a$ and $xa^2f=x'a^2$ if $f\in U\cup U^{-1}$.
  \end{enumerate}
  Then $G$ acts well on $X'$.
\end{lemma}
\begin{proof} 
First observe that this clearly defines a partial action in the sense of Definition \ref{d:partial} and that the actions of $a$ and $t$ still satisfy conditions 1. and 2. of Definition \ref{d:good}.
  
  For the remaining conditions of \ref{d:good} 
  it suffices to prove that if a cyclically reduced word $w$ leaves a triple $C$ in $X'$
  invariant, then $w\in G_0$ (and hence $w$ is conjugate to $a$ or $t$ by assumption) or $C$ and the (possibly braided) cycle described
  by $w$ applied to $C$ are contained in $X$. Since the previous
  action was good, this is enough.

  Suppose otherwise: let $w\in G\setminus G_0$ be cyclically reduced (in the sense of \ref{d:normalform})  leaving the triple $C$ invariant. Assume that at least one triple of the cycle given by applying $w$ to $C$ does not  belong to $X$. 
  
First assume that there is a triple in the cycle which does not belong to $X$, but both its neighbours do. This easily implies that a cyclic permutation of $w$ contains
the subword $ff^{-1}$  as $f$ is the only element  taking a triple
from $X$ to a triple not entirely belonging to $X$. So $w$ is
not cyclically reduced, a contradiction.

Next assume that there are two neighbouring triples $C_1,C_2$
  in the cycle which do not belong to $X$. Then by the properties of a cyclically reduced word and the
  definition of $X'$, $C_1$ and $C_2$ are connected by
  an element $g\in  G_0\setminus 1$. 
So the cycle contains a segment $(B,C_1,C_2,D)$ where the triples $B,D$ are contained in $X$ and necessarily $Bf=C_1, C_1g=C_2$ and $C_2f^{-1}=D$.

Then  $f\notin R\cup R^{-1}$ as $G_0$ acts regularly on $x'G_0$.
  
  If $f\in S\cup S^{-1}$, we must have $g=t$ since  on $x'G_0$ the element
 $f^{-1}$ is only defined on $x't$.
  So a cyclic permutation of
  $w$ contains the subword  $f\cdot t\cdot
  f^{-1}$, a contradiction.

   Similarly, if $f\in U\cup U^{-1}$, we have $g=a^{\pm 1}$ 
   and a cyclic permutation of
  $w$ contains the subword  $f\cdot a^{\pm 1}\cdot
  f^{-1}$, again  a contradiction.

This shows that if a triple $C$ becomes invariant under some $g\in G$
under the extended action, this is induced by conjugation under the
previous action. Hence Condition 4. of \ref{d:good} is preserved
as well. 
\end{proof}

We next show how to extend the group action in order to connect the unique  triple $A=(x_0,x_0a,x_0a^2)$ with $x_0at=x_0a^2$
to a triple $B$ where $B$ is either an $a$-triple, a $t$-triple
or a free triple:

\begin{lemma}[Connecting $A$ to other triples]\label{l:join}
  Assume that $G$ acts well on $X$,
  let $A=(x_0,x_0a,x_0a^2)=(x_0,x_0a,x_0at)$ be as before 
  and let $B$ be an $a$-, $t$- or a free triple
  for which there is no $g\in G$ with $Ag=B$.  
  Let $f\in R\cup S\cup U$ 
  be an element which does not yet act anywhere with
  \begin{enumerate}
  \item $f\in U$ if $B$ is an $a$-triple;
  \item $f\in S$ if $B$ is a $t$-triple;
  \item $f\in R$ if $B$  is a free triple
  
  \end{enumerate}  
  Extend the action by setting $Af=B$. Then this action of $G$ on $X$
  is again good.
\end{lemma}  
\begin{proof} 
Again it is clear that this defines a partial action
in the sense of Definition \ref{d:partial} and that 1. and 2.
of Definition \ref{d:good} continue to hold.
Also note that
 since the setwise stabilizer of $A$ in $G$
  is $S_3$ the assumptions imply that there is no $w\in G$ not containing
  $f,f^{-1}$ taking $A$ to $B$ as a set.  

 To prove the remaining conditions of \ref{d:good} let $w$ be a cyclically reduced word (in the sense of \ref{d:normalform}) leaving  some triple $C\in (X)^3$ invariant.
 Since the previous action was good, it suffices to show that $w$ does not contain $f$
 or $f^{-1}$. Suppose otherwise. Then by
 cyclically permuting $w$ and taking inverses we may assume
 that $w=f\cdot w'$. So $w$ stabilizes $A$ and $w'$ takes $B$ to $A$ as
 a set.
 By assumption on $A,B$ the subword $w'$ must contain $f$.
 Hence we may write $w'=v'\cdot f^\epsilon v$ for some subword $v'$ not containing $f$ or $f^{-1}$. 
 We distinguish two cases:
  \begin{enumerate}
  \item $\epsilon=1$. Then $v'$ takes $B$ to $A$ as a set as $f$ is only defined on $A$.
  Since $v'$ does not contain $f,f^{-1}$, this contradicts the assumption on $A,B$.
  \item $\epsilon=-1$. Then  $v'$  leaves $B$    invariant. Since the previous action was good, we either have $v'=t$ and $f\in S$ or $v'=a^{\pm 1}$ and $f\in U$ (according to whether $B$ is an $a$- or a $t$-triple). In either case  $v'$ commutes with $f$, contradicting the assumption that $w$ be cyclically reduced.
  \end{enumerate}\end{proof}

\begin{corollary}
  Assume that $G$ acts well on $X$ with $|X|\leq \max\{\aleph_0,|G|\}$
  and there are sufficiently many elements of $R, S$ and $U$ whose action
  is not yet defined anywhere. Then we can extend the partial action
  of $G$ on $X$ to a sharply $3$-transitive (total) action on some appropriate
  superset $Y$.
\end{corollary}
\begin{proof}
  Fix the unique $a$-triple $A=(x_0,x_0a,x_0a^2)=(x_0,x_0a,x_0at)$ 
  in $X_0$. Using the previous lemmas we define a
  set $Y$ and a $3$-sharp action of $G$ on $Y$ with the following
  properties:
  \begin{enumerate}
  \item all $a$-triples are connected to $A$;
  \item all $t$-triples  are connected to $A$;
  \item any triple $(x,y,z)$ can be shifted by an element of $G$.
  \end{enumerate}
  The last property can be achieved using Lemma \ref{l:join}:
  suppose $B=(x,y,z)$ cannot be shifted by an element of $G$ at a certain  stage of the construction. Then $B$ is a free triple:
otherwise an element of $G$ leaving $B$ invariant
would have to be an involution $t'\in G$.
  Since $t'=hth^{-1}$ for some $h\in G$, the triple
  $Bh$ is a $t$-triple and hence connected to $A$, making $B$ shiftable
  as well.
  
  Thus, $B$ is a free triple at that stage, and
  we later extend the action by putting $Ar=B$ for some $r\in R$.
  Then $B$  can be shifted by $a^r$.

  This easily implies that the action of $G$ on $Y$ is sharply
  $3$-transitive, i.e.  all triples are connected to
  $A$:  let $B$ be a triple and $g\in G$ shift $B$.  Then we have
  $g=hah\inv$ for some $h\in G$,
  so $Bh$ is an $a$-triple
  and whence connected to $A$.
\end{proof}

\noindent This concludes the proof of Theorem \ref{t:main}.




\vspace{.5cm}

\noindent\parbox[t]{15em}{
Katrin Tent,\\
Mathematisches Institut,\\
Universit\"at M\"unster,\\
Einsteinstrasse 62,\\
D-48149 M\"unster,\\
Germany,\\
{\tt tent@wwu.de}}

\end{document}